\documentclass[12pt,reqno]{amsart}
\usepackage{hyperref,cleveref,amssymb,eqnarray,mathrsfs,xcolor}

\theoremstyle{plain}
\newtheorem{theorem}{Theorem}[section]
\newtheorem{proposition}[theorem]{Proposition}
\newtheorem{corollary}[theorem]{Corollary}
\newtheorem{lemma}[theorem]{Lemma}

\theoremstyle{definition}

\newtheorem{example}[theorem]{Example}

\newcommand{\term}[1]{{\textit{\textbf{#1}}}}   

\newcommand{\Pelc}[1]{\left(\Sigma X_n\right)_p}

\DeclareSymbolFont{bbold}{U}{bbold}{m}{n}
\DeclareSymbolFontAlphabet{\mathbbold}{bbold}

\DeclareMathOperator{\orth}{Orth}
\renewcommand{\le}{\leqslant}

\begin{document}


\title[On diagonal of Riesz operators]{On the diagonal of Riesz operators on Banach lattices\\
\vskip.5cm
\small{D\lowercase{edicated to the memory of}  W\lowercase{im} L\lowercase{uxemburg}}}

\author{R. Drnov\v sek}
\address{Faculty of Mathematics and Physics, University of Ljubljana,
  Jadranska 19, 1000 Ljubljana, Slovenia \ \ \ and \ \ \
  Institute of Mathematics, Physics, and Mechanics, Jadranska 19, 1000 Ljubljana, Slovenia}
\email{roman.drnovsek@fmf.uni-lj.si}

\author{M. Kandi\'c}
\address{Faculty of Mathematics and Physics, University of Ljubljana,
  Jadranska 19, 1000 Ljubljana, Slovenia \ \ \ and \ \ \
  Institute of Mathematics, Physics, and Mechanics, Jadranska 19, 1000 Ljubljana, Slovenia}
\email{marko.kandic@fmf.uni-lj.si}


\date{\today}

\begin{abstract}
This paper extends the well-known Ringrose theory for compact operators to polynomially Riesz operators on Banach spaces. The particular case of an ideal-triangularizable Riesz operator on an order continuous Banach lattice yields that the spectrum of such operator lies on its diagonal, which motivates the systematic study of an abstract diagonal of a regular operator on an order complete vector lattice $E$. We prove that the class $\mathscr D$ of regular operators for which the diagonal coincides with the atomic diagonal is always a band in $\mathcal L_r(E)$, which contains the band of abstract integral operators. If $E$ is also a Banach lattice, then $\mathscr D$ contains positive Riesz operators.
\end{abstract}

\maketitle

\section{Introduction}
Although the spectral theory of Riesz operators is the same as that of compact operators, they behave differently. For example, the sum and the product of two Riesz operators can be something other than a Riesz operator, which is not the case for compact operators. The abundance of invariant subspaces for compact operators and Ringrose's theorem guarantee that every compact operator admits an abstract upper-triangular form which is a generalization of Schur's form from linear algebra locating the spectrum of a matrix on its diagonal. Since Riesz operators possibly do not have non-trivial invariant subspaces, they generally do not admit such forms. This paper aims first to develop Ringrose's theory (\Cref{Section: Ringrose}) for polynomially Riesz operators and then apply it to study some aspects of Riesz operators on Banach lattices. As a result, we obtain that the spectrum of an ideal-triangularizable Riesz operator can be realized as the set of its diagonal coefficients, which motivates the systematic study (\Cref{section:diagonals}) of the abstract diagonal of a given regular operator on an order complete vector lattice $E$. We prove that the set of all regular operators for which the diagonal agrees with the atomic diagonal is a band in
the order complete vector lattice of all regular operators $\mathcal L_r(E)$, which always contains the band of all abstract integral operators. If, in addition, $E$ is a Banach lattice, this set contains every positive Riesz operator.

\section{Preliminaries}\label{preliminaries}

Throughout this paper, $X$ will denote a complex Banach space and $\mathcal B(X)$ the Banach algebra of all bounded linear operators on $X$. The closed ideal of all compact operators on $X$ will be denoted by $\mathcal K(X)$. The spectrum of an element $a$ in a
complex Banach algebra $\mathcal A$ is denoted by $\sigma(a)$. The spectrum of a bounded linear operator $T\in \mathcal B(X)$ is the spectrum of $T$ considered as an element of the Banach algebra $\mathcal B(X)$. The spectral radius of $T$ is denoted by $\rho(T)$. By $\pi\colon \mathcal B(X)\to \mathcal B(X)/\mathcal K(X)$ we denote the canonical projection onto the
\term{Calkin algebra} $\mathcal B(X)/\mathcal K(X)$. The \term{essential spectrum} of $T \in \mathcal B(X)$ is defined by $\sigma_e(T) := \sigma(\pi(T))$.

Let $T\in \mathcal B(X)$ be a bounded linear operator on $X$. If $\pi(T)$ is a quasinilpotent element of the Calkin algebra, then $T$ is called a \term{Riesz operator}. Hence, $T$ is a Riesz operator if $\sigma_e(T)=\{0\}$. It should be noted that Dowson
\cite[Definition 3.4]{Dowson78} defines Riesz operators as those operators which share the spectral theory of compact operators. Aforementioned definitions of a Riesz operator are equivalent by \cite[Theorem 3.12]{Dowson78}. Hence, the class of Riesz operators on $X$ always contains the ideal $\mathcal K(X)$ of compact operators. Since the image of every power compact operator in the Calkin algebra is nilpotent, the class of Riesz operators also contains every power compact operator.
If there exists a non-zero complex polynomial $p$ such that $p(T)$ is a Riesz operator, then $T$ is called a \term{polynomially Riesz} operator. We refer the reader  to \cite{AA02} for more details about operators on Banach spaces.

The letter $E$ shall stand for an Archimedean vector lattice. By $E^\sim$ and $E_n^\sim$, we denote the order dual and the order continuous dual of $E$, respectively. If $E$ is a normed lattice, then the norm dual $E^*$ of $E$ is contained in $E^\sim$. If, in addition, $E$ is a Banach lattice, then $E^\sim=E^*$. Let $T\colon E\to F$ be a positive operator between vector lattices.
The \term{absolute kernel} of the operator $T$ is defined by $N_T  = \{x\in E:\; T|x|=0\}$.
It should be noted that $N_T$ is always an ideal of $E$. If $T$ is also order continuous, then $N_T$ is a band in $E$. The disjoint complement $C_T:=N_T^d$ of $N_T$ is called the \term{carrier} of $T$. A band $B$ in $E$ is a \term{projection band} if $E=B\oplus B^d.$ The band projection $P_B$ of $E$ onto the projection band $B$ is always an order continuous positive operator. An atom $a$ of a vector lattice $E$ is a positive vector $a$ with the following property: if $0\leq x,y\leq a$ and $x\perp y$, then $x=0$ or $y=0$. Since $E$ is Archimedean, the notion of an atom is equivalent to the notion of a discrete vector saying that the linear span of $a$ is a projection band in $E$. The \term{atomic part} $A$ of $E$ is the band generated by the set of all atoms in $E$. Its disjoint complement $C:=A^d$ is called the \term{non-atomic part} or the \term{continuous part} of $E$. If $E$ is order complete, then every band is a projection band which yields the order decomposition $E=A\oplus C$ of the vector lattice $E$.


Given $F \subseteq E^\sim$ a subspace, the symbol $E\otimes F$ denotes the vector space of all finite-rank operators which can be expressed as a finite sum of rank one operators of the form $x\otimes \varphi$ with $x \in E$ and $\varphi \in F$. The band generated by $E\otimes E_n^\sim$ in $\mathcal L_r(E)$ is called the band of \term{abstract integral operators}. If every band of $E$ is invariant under $T\colon E\to E$, then $T$ is a \term{band preserving} operator. Order bounded band preserving operators are called \term{orthomorphisms}. In general, band preserving operators are not order bounded. Abramovich, Veksler and Koldunov proved in \cite{AVK} (see also
\cite[Theorem 4.76]{Aliprantis:06})  that every band preserving operator on a Banach lattice is order bounded, and hence, an orthomorphism.   Since $E$ is an Archimedean vector lattice, the set of orthomorphisms $\orth(E)$ on $E$ is a vector lattice under pointwise algebraic and lattice operations (see \cite{BK, CD}). In fact, for each $T \in \orth(E)$ and for each vector $x\in E$ we have
\begin{equation}\label{modulus}
|T|(|x|)=\big|T|x|\big|=|Tx|.
\end{equation}
By \cite[Theorem 2.44]{Aliprantis:06} (see also \cite[Theorem 1.3]{LS78}) every orthomorphism on an Archimedean vector lattice is order continuous, and therefore, order bounded by \cite[Theorem 2.1]{AS05}. If, in addition, $E$ is order complete, then by \cite[Theorem 2.45]{Aliprantis:06} the vector lattice $\orth(E)$ is a band in $\mathcal L_r(E)$ generated by the identity operator. Since $\mathcal L_r(E)$ is also order complete, the band of orthomorphisms is a projection band in $\mathcal L_r(E)$. We refer the reader to \cite{AA02, Aliprantis:06, Luxemburg71} for more details about vector and Banach lattices and operators acting on them.

We say that a chain $\mathcal C$ of closed subspaces of a Banach space $X$ is a \term{complete} chain if it contains arbitrary intersections and closed linear spans of its members. If a closed subspace $\mathcal M$ is in a complete chain $\mathcal C$, then the {\term{predecessor}
$\mathcal M_{-}$ of $\mathcal M$ in $\mathcal C$ is defined as the closed linear span of all proper subspaces of $\mathcal M$ belonging to $\mathcal C$.
Every maximal chain $\mathcal C$ of closed subspaces of $X$ is complete and, for each subspace $\mathcal M$ in $\mathcal C$,
the dimension of the space quotient space $\mathcal M/\mathcal M_{-}$ is at most one.

A family $\mathcal F$ of operators on a Banach space $X$ is \term{reducible} if there exists a non-trivial closed subspace of $X$ that is invariant under every operator from $\mathcal F$. If there exists a maximal chain $\mathcal C$ of closed subspaces of $X$ such that every subspace from the chain $\mathcal C$ is invariant under every operator from $\mathcal F$, then $\mathcal F$ is said to be \term{triangularizable}, and $\mathcal C$ is called a
\term{triangularizing chain} for $\mathcal F$. A family $\mathcal F$ of operators on a Banach lattice $E$ is said to be \term{ideal-reducible} if there exists a non-trivial closed ideal of $E$ that is invariant under every operator from $\mathcal F$. Otherwise, we say that $\mathcal F$ is
\term{ideal-irreducible}.
A family  $\mathcal F$ of operators on a Banach lattice is said to be
\term{ ideal-triangularizable} if it is triangularizable and at least one of (possibly many) triangularizing chains of $\mathcal F$ consists of closed ideals of $E$. For a more detailed exposition on triangularizability we refer the reader to \cite{RaRo}.

\section{Ringrose type theorems for polynomially Riesz operators}\label{Section: Ringrose}

It is a well-known fact from linear algebra that every complex $n\times n$ matrix is similar to an upper-triangular matrix whose set of diagonal entries coincides with its spectrum.
Ringrose \cite{Ringrose} proved that the spectrum of a given compact operator $T$ on a complex Banach space could also be realized in terms of something resembling diagonal coefficients as follows. By \cite[Theorem 1]{Ringrose}, the operator $T$ is triangularizable. Pick any triangularizing chain $\mathcal C$ for $T$. Then for each $\mathcal M\in \mathcal C$, we have either $\mathcal M=\mathcal M_-$ or the quotient Banach space $\mathcal M/\mathcal M_-$ is one-dimensional. If $\mathcal M\neq \mathcal M_-$, then $T$ induces the operator $T_{\mathcal M}$ on the one-dimensional Banach space $\mathcal M/\mathcal M_-$ given as $T_{\mathcal M}(x+\mathcal M_-)=\lambda_{\mathcal M}(x+\mathcal M_-)$ for any $x\in \mathcal M\setminus \mathcal M_-$. Hence, $T_{\mathcal M}$ acts on $\mathcal M/\mathcal M_-$ as a multiplication by the uniquely determined number $\lambda_{\mathcal M}$ which is called the \term{diagonal coefficient} of $T$ with respect to $\mathcal M$. Ringrose proved (see \cite[Theorem 2]{Ringrose}) that
$$\sigma(T)\setminus \{0\}=\{\lambda_{\mathcal M}:\; \mathcal M\neq \mathcal M_-\}\setminus \{0\}.$$ Furthermore, the diagonal multiplicity of $\lambda\in \sigma(T)\setminus \{0\}$ agrees with its algebraic multiplicity.
Konvalinka extended Ringrose's theorem (see \cite[Theorem 2.7]{Konvalinka}) to polynomially compact operators. It should be noted that Konvalinka \cite{Konvalinka} did not consider the part of Ringrose's theorem regarding multiplicities.

In this section, we first extend Ringrose's theorem to the class of polynomially Riesz operators and then obtain a version for ideal-triangularizable Riesz operators on order continuous Banach lattices. Before we proceed to the main results concerning Ringrose's theory for polynomially Riesz operators, we extend the well-known spectral theory of Riesz operators (see \cite{AA02}) to polynomially Riesz operators, which will be needed in the paper. Although the results should be known, we include our proofs for the sake of the reader.

\begin{proposition}\label{Minimal polynomial polynomially Riesz}
Let $T$ be a polynomially Riesz operator on an infinite-dimensional Banach space. Then there exists the unique monic polynomial $m$ of the smallest degree such that $m(T)$ is a Riesz operator and that $m$ divides every polynomial $p$ for which $p(T)$ is a Riesz operator.
\end{proposition}

\begin{proof}
  It should be clear that there exists at least one monic polynomial $p$ such that $p(T)$ is a Riesz operator. Let $m$ be such a polynomial with the smallest possible degree. Let $p$ be any polynomial for which $p(T)$ is a Riesz operator. Then there exist polynomials $k$ and $r$ such that the degree of $r$ is strictly smaller than that of $m$ and that $p=km+r$. Suppose that $r$ is non-zero. Since $k(T)$ and $m(T)$ commute, the images of $p(T)$ and $k(T)m(T)$ in the Calkin algebra are quasinilpotent elements. Furthermore, since $p(T)$ and $k(T)m(T)$ commute, the equality $r(T)=p(T)-k(T)m(T)$ together with the same argument as above shows that $r(T)$ is a Riesz operator. This is in contradiction with the minimality of the degree of $m$.
\end{proof}

The polynomial $m$ from Proposition \ref{Minimal polynomial polynomially Riesz} is called the {\it minimal polynomial} of $T$.

\begin{proposition}\label{Essential spectrum Riesz operator}
Let $T$ be a polynomially Riesz operator on a Banach space $X$. Let $m(z)=(z-\lambda_1)^{j_1}\cdots (z-\lambda_k)^{j_k}$ be the minimal polynomial of $T$. Then
$$ \sigma_{e}(T)=m^{-1}(\{0\}) = \{\lambda_1, \ldots, \lambda_k\} . $$
Each $\lambda \in \sigma(T)\setminus \sigma_e(T)$ is an isolated point of $\sigma(T)$, it is an eigenvalue of $T$ and is a pole of the resolvent of $T$. Furthermore, the spectral projection $P_{\lambda}(T)$ is of finite rank, and
the operator $\lambda-T$ has finite ascent and descent.
 \label{polyRiesz}
\end{proposition}

\begin{proof}
  The essential spectrum of the operator $T$ is the spectrum of the element $T+\mathcal K(X)$ in the Calkin algebra $\mathcal B(X)/\mathcal K(X)$. Since the operator $m(T)$ is a Riesz operator, the element $m(T)+\mathcal K(X)$ in the Calkin algebra is quasinilpotent. From the spectral mapping for polynomials, we conclude that
  $$m(\sigma(T+\mathcal K(X))=\sigma(m(T+\mathcal K(X))=\sigma(m(T)+\mathcal K(X))=\{0\}$$ from where it follows that
  $$\sigma_e(T)=\sigma(T+\mathcal K(X))\subseteq m^{-1}(\{0\}).$$
  Suppose that $\sigma_e(T)$ is a proper subset of $m^{-1}(\{0\})$. Then there exists a root $\lambda$ of the polynomial $m$ such that the image of the operator $T-\lambda I$ is invertible in the Calkin algebra. If we write $m(z)-m(\lambda)=(z-\lambda)q(z)$, we have
  $$m(T)=m(T)-m(\lambda)I=(T-\lambda I)q(T)$$ from where we conclude $\pi(q(T))=(\pi(T-\lambda I))^{-1}\pi(m(T))$. Since $\pi((T-\lambda I))^{-1}$ and $\pi(m(T))$ commute, and since the image of $\pi(m(T))$ is quasinilpotent, the image of the operator $q(T)$ is quasinilpotent in Calkin algebra as well. Hence, $q(T)$ is a Riesz operator. This is in contradiction with the fact that $m$ is the minimal polynomial of $T$. This proves that $\sigma_{e}(T)=m^{-1}(\{0\})$.

Now, take $\lambda \in \sigma(T)\setminus \sigma_e(T)$.
By \cite[Lemma 7.43]{AA02}, $\lambda$ is an isolated point of $\sigma(T)$ since it can be joined with any point of the
resolvent set of $T$ by a path lying outside a finite set $\sigma_{e}(T)$. The remaining assertions follow from \cite[Theorem 7.44]{AA02}.
\end{proof}

The following corollary follows from Proposition \ref{polyRiesz} in a similar way as \cite[Theorem 7.48]{AA02} gives \cite[Corollaries 7.50 and 7.51]{AA02}.

\begin{corollary}
The spectrum of a polynomially Riesz operator is at most countable, and
its resolvent set is connected.
\end{corollary}

Now we are finally able to extend the well-known Ringrose's theorem for block triangularizations (see \cite[Theorem 7.2.7]{RaRo}) to the class of polynomially Riesz operators.

\begin{theorem} \label{riesz trikot haha}
	Let $T$ be a polynomially Riesz operator on a Banach space $X$. Then for every complete chain $\mathcal C$ of closed invariant subspaces of the operator $T$, we have
	$$ \sigma(T)=\sigma_e(T) \cup \bigcup \{\sigma(T_{\mathcal M}):\; \mathcal M\in \mathcal C,\, \mathcal M_-\neq \mathcal M\}. $$
\end{theorem}

\begin{proof}
	Suppose that $\lambda \in \sigma(T_{\mathcal M}) \setminus \sigma_e(T)$ for some $\mathcal M\in \mathcal C$ satisfying $\mathcal M\neq \mathcal M_-.$ It follows from \cite[Theorem 3.23]{Dowson78} that the operator $T_{\mathcal M}$ is a polynomially Riesz operator on $\mathcal M/\mathcal M_-$. Note that by \Cref{Minimal polynomial polynomially Riesz} the minimal polynomial of $T_{\mathcal M}$ divides
the minimal polynomial $m$ of $T$, and so $\sigma_e(T_{\mathcal M}) \subseteq \sigma_e(T)$ by \Cref{Essential spectrum Riesz operator}.

	 We claim that $T_{\mathcal M}-\lambda I$ is not surjective on $\mathcal M/\mathcal M_-$.
Since $\sigma_e(T_{\mathcal M}) \subseteq \sigma_e(T)$, we have
$\lambda \in \sigma(T_{\mathcal M}) \setminus \sigma_e(T_{\mathcal M})$, and the operator $T_{\mathcal M}-\lambda I$ has a finite ascent and a finite descent. \cite[Lemma 2.21]{AA02} implies that its ascent and descent are the same (we denote them by $p\in\mathbb N$), and also the following equality holds
	 $$\mathcal M/\mathcal M_-=\ker(T_{\mathcal M}-\lambda I)^p \oplus \mathrm{im}(T_{\mathcal M}-\lambda I)^p.$$
	If $T_{\mathcal M}-\lambda I$ is surjective, then $p=1$, so that $\ker(T_{\mathcal M}-\lambda I)=\{0\}$. This implies that $T_{\mathcal M}-\lambda I$ is also injective, which is a contradiction since $\lambda\in \sigma(T_{\mathcal M}).$

 Therefore, the operator $T|_{\mathcal M}-\lambda I$ also cannot be surjective. Since the operator $T|_{\mathcal M}$ is also a polynomially Riesz operator by \cite[Theorem 3.21]{Dowson78},  similarly, as above, we can see that $T|_{\mathcal M}-\lambda I$ is not injective. This implies that $\lambda$ is in the point spectrum of the operator $T$.
	
	Choose $\lambda\in \sigma(T) \setminus \sigma_e(T)$. Since the spectral projection $P_\lambda(T)$ is of finite-rank, the kernel $\ker (T-\lambda I)$ which is contained in the range of $P_\lambda(T)$ is finite-dimensional, so that the set
	$$\mathcal O=\{x\in X:\; Tx=\lambda x, \|x\|=1\}$$ is a compact subset of $X$. By the finite intersection property of compact sets, the subspace
	$$\mathcal M=\bigcap\{\mathcal N:\, \mathcal N\cap \mathcal O\neq \emptyset,\, \mathcal N\in\mathcal C\}$$
has a non-empty intersection with $\mathcal O$, and so it is non-trivial. Since $\mathcal C$ is complete, we have $\mathcal M\in\mathcal C$.

We claim that $\mathcal M\neq \mathcal M_-$ and $\lambda\in \sigma(T_{\mathcal M}).$ For every proper subspace $\mathcal L\in \mathcal C$ of $\mathcal M$ we can prove similarly as in \cite[Theorem 2.7]{Konvalinka}
	that for the descent $N$ of $T-\lambda I$ we have $\mathcal L\subseteq \mathrm{im}(\lambda I-T)^N.$ Since $\mathcal M_-$ is spanned by proper subspaces of $\mathcal M$, it follows $\mathcal M_-\subseteq \mathrm{im}(\lambda I-T)^N$. Since the intersection of $\mathcal M$ with $\ker(\lambda I-T)^N$ is non-trivial, we have $\mathcal M\neq \mathcal M_-$. If $x\in \mathcal M\backslash\mathcal M_-$ satisfies $Tx=\lambda x$, then $T_{\mathcal M}(x+\mathcal M_-)=\lambda(x+\mathcal M_-)$, so that $\lambda\in \sigma(T_{\mathcal M}).$
\end{proof}

The following example shows that in Theorem \ref{riesz trikot haha} we cannot omit the assumption that $T$ is a polynomially Riesz operator.

\begin{example}
Let $T$ be the adjoint of the shift operator on $\ell^2$, that is,
$$T(x_1, x_2, x_3, \ldots) = (x_2, x_3, x_4, \ldots).$$
Let $\{e_k\}_{k \in \mathbb N}$ be the standard basis of $\ell^2$.
For each $n \in \mathbb N$ let ${\mathcal M}_n$ be the linear span of vectors $e_1$, $e_2$, $\ldots$, $e_n$.
Set also ${\mathcal M}_0 = \{0\}$. Then
$$ \mathcal C = \{ \mathcal M_n : n \in \mathbb N \cup \{0\} \} \cup \{\ell^2\} $$
is a complete chain of closed invariant subspaces for the operator $T$.
Since $T(\mathcal M_n) = \mathcal M_{n-1}$,  we have $T_{\mathcal M_n} = 0$.
Therefore,
$$\sigma_e(T) \cup \bigcup \{\sigma(T_{\mathcal M}):\; \mathcal M\in \mathcal C,\, \mathcal M_-\neq \mathcal M\}$$
is equal to $\{\lambda \in \mathbb C : | \lambda | = 1\} \cup \{0\}$, while $\sigma(T)$ is equal to the closed unit disk
$\{\lambda \in \mathbb C : | \lambda | \le 1\}$.
\end{example}

The following corollary is an immediate consequence of Theorem \ref{riesz trikot haha}.

\begin{corollary}[Ringrose's theorem for polynomially Riesz operators]\label{cor to main 1}
Let $T$ be a triangularizable polynomially Riesz operator on a Banach space $X$. Then for every triangularizing chain
$\mathcal C$ of the operator $T$ we have
$$\sigma(T)=\sigma_e(T) \cup \{\lambda_{\mathcal M}:\, \mathcal M\in \mathcal C,\, \mathcal M\neq \mathcal M_-\}.$$
\end{corollary}

Positive ideal-triangularizable Riesz operators have positive spectra.

\begin{corollary}
Let $T$ be an ideal-triangularizable positive operator on a Banach lattice $E$.
\begin{enumerate}
\item[(i)] If $T$ is a polynomially Riesz operator, then $\sigma(T) \setminus \sigma_e (T) \subseteq [0, \rho (T)]$.
\item[(ii)] If $T$ is a Riesz operator, then $\sigma(T) \subseteq [0, \rho (T)]$.
\end{enumerate}
\end{corollary}

\begin{proof}
(i) Let $\mathcal C$ be an ideal-triangularizing chain for the operator $T$. Pick any $\lambda \in \sigma(T)\setminus \sigma_e(T)$. By \Cref{cor to main 1}, there exists an ideal $\mathcal J\in \mathcal C$ such that the dimension of $\mathcal J/\mathcal J_-$ is one and $\lambda=\lambda_{\mathcal J}$.
Hence, $T_{\mathcal J}=\lambda_{\mathcal J}I$ on $\mathcal J/\mathcal J_-$. Since $T$ is positive, the induced operator $T_{\mathcal J}$ is also positive, so that $\lambda_{\mathcal J}\geq 0$.

(ii) follows from (i) and the fact that the essential spectrum of a Riesz operator equals $\{0\}$.
\end{proof}

A close inspection of the proof of \cite[Theorem 7.2.9]{RaRo} shows that we have the following sharpening of Ringrose's theorem. To avoid repetition, we omit its proof.

\begin{theorem}\label{RingroseMultiplicity}
Let $T$ be a triangularizable Riesz operator on a Banach space $X$.
The diagonal multiplicity of each non-zero eigenvalue of $T$ with respect to any triangularizing chain is equal to its algebraic multiplicity.
\label{multiplicities}
\end{theorem}


We conclude this section by \Cref{spectrum is on a real diagonal}, which roughly says that the spectrum of an ideal-triangularizable Riesz operator lies on its ``matrix" diagonal, which we shall explain in the following lines.
Let $a$ be an atom in $E$. Then for each vector $x\in E$ we have the order decomposition $x=\lambda_xa+y$ for uniquely determined $\lambda_x \in \mathbb R$ and $y\in \{a\}^\perp$. The positive linear functional $\varphi_a\colon E\to \mathbb R$ defined as $\varphi_a(x)=\lambda_x$ is a positive lattice homomorphism called the \term{coordinate functional} associated with the atom $a$. For an operator $T$ on $E$, the number $\varphi_a(Ta)$ is called the \term{diagonal entry} of $T$ with respect to the atom $a$. If $a'$ is any other atom in $E$ whose linear span equals the linear span of $a$, then $a$ and $a'$ are called \term{equivalent}. Since for equivalent atoms $a$ and $a'$ we have
$$\varphi_a(Ta)=\varphi_{a'}(Ta'),$$
the notion of the diagonal entry is independent of the choice of equivalent atoms.
Let $\mathcal A$ be any maximal set of pairwise non-equivalent atoms.

The following result, which is an improvement of \cite[Proposition 4]{Kandic13} yields that the spectrum of an ideal-triangularizable Riesz operator on an order continuous Banach lattice equals to the set of its diagonal entries counted with multiplicities.

\begin{theorem}\label{spectrum is on a real diagonal}
Let $T$ be an ideal-triangularizable Riesz operator on an order continuous Banach lattice $E$. Then
\begin{equation}\label{veckratnost diagonalcev}
 \sigma(T) \setminus \{0\} = \{ \varphi_a(T a) : a\in \mathcal A  \} \setminus \{0\}
\end{equation}
according to multiplicities of non-zero eigenvalues of the operator $T$.
\end{theorem}

\begin{proof}
Let $\mathcal C$ be a triangularizing chain for $T$.  Assume that $\lambda \in \sigma(T) \setminus \{0\}$. We first show that  $\lambda = \varphi_a(T a)$ for some atom $a \in E$.  By Ringrose's theorem for polynomially Riesz operators (see \Cref{cor to main 1})  there is  $\mathcal J \in \mathcal C$ such that
$\lambda = \lambda_{\mathcal J}$. Since  $\mathcal J/ \mathcal J_-$ is one-dimensional and $E$ is order continuous, there exists an atom $a$ such that $\mathcal J = \mathcal J_- \oplus \mathbb R a$. From $Ta-\lambda_\mathcal Ja\in \mathcal J_-$ we conclude
$\varphi_a(Ta-\lambda_\mathcal Ja)=0$ or equivalently $\varphi_a(Ta)=\lambda_{\mathcal J}\varphi_a(a)=\lambda_{\mathcal J}$.

Assume that the multiplicity of $\lambda$ is $k>1$. We must show that  $\lambda = \varphi_{a_j}(T a_j)$ for some atoms $a_1,\ldots,a_k \in E$ with norm one.  By \Cref{RingroseMultiplicity}  there is a chain $\mathcal J_1\subsetneq \cdots \subsetneq \mathcal J_k$ in $\mathcal C$ such that
$\lambda=\lambda_{\mathcal J_j}$ for each $1\leq j\leq k$. By the first part of the proof for each $1\leq j\leq k$ there exists an atom $a_j$ of norm one such that $\mathcal J_j=(\mathcal J_j)_-\oplus \mathbb R a_j$ and $\lambda=\varphi_{a_j}(Ta_j)$. If $i<j$, then $a_i\in \mathcal J_i\subsetneq \mathcal J_{j}$ yields that $a_i\in (\mathcal J_{j})_-$, and so $a_i\neq a_j$.


To prove the converse inclusion, let $a$ be an atom of $E$ such that $\varphi_a(T a) \neq 0$. Define
$\mathcal F = \{ \mathcal J \in \mathcal C : a \in \mathcal J \} $. Let $\mathcal J_a$ be the intersection of all members of
$\mathcal F$. Then $\mathcal J_a$ is an ideal from $\mathcal C$ and $a \in \mathcal J_a$. It is easy to see that
$a \not\in (\mathcal J_a)_-$ and that $\varphi_a(T a)$ belongs to  $\sigma(T) \setminus \{0\}$.
\end{proof}

The preceding theorem can be applied in the case of the classical spaces $\ell^p$ $(1\leq p<\infty)$ and $c_0$ to prove that the non-zero
spectrum counted according to multiplicities of eigenvalues of an ideal-triangularizable Riesz operator lies on the diagonal of the associated matrix with respect to the standard basis. Indeed, the standard basis vectors of $\ell^p$ ($1\leq p<\infty)$ and $c_0$ are precisely atoms of norm one.

\section {The diagonal and the atomic diagonal}\label{section:diagonals}

Let $T$ be an ideal-triangularizable Riesz operator on an order continuous Banach lattice $E$. By \Cref{spectrum is on a real diagonal}, we have
\begin{equation}\label{veckratnost diagonalcev}
 \sigma(T) \setminus \{0\} = \{ \varphi_a(T a) : a\in \mathcal A  \} \setminus \{0\}
\end{equation}
according to multiplicities of non-zero eigenvalues of the operator $T$.
If we consider the $2\times 2$ operator matrix associated with $T$ with respect to the decomposition $E=A\oplus C$, then the spectrum of $T$ is contained on the ``diagonal" of $T$. In this section, we study the abstract diagonal of a regular operator on an order complete vector lattice. We prove that every diagonal splits into its atomic and its continuous part. Moreover, we prove that the continuous part of the diagonal of an abstract integral operator or a Riesz operator is always zero. This motivates the study of those regular operators for which the diagonal is equal to its atomic part.

Throughout this section, we assume that $E$ is an order complete vector lattice. Then the band $\orth(E)$ is a projection band in $\mathcal L_r(E)$ generated by the identity operator, and so we have the following order decomposition
$$\mathcal L_r(E)=\orth(E)\oplus \orth(E)^{d}=\{I\}^{dd}\oplus \{I\}^d$$
of the order complete vector lattice $\mathcal L_r(E)$.

  Let ${\mathcal P}_E$ be the band projection onto the projection band $\orth(E)$. Given a regular operator $T\in \mathcal L_r(E)$, the projection $\mathcal P_E(T)$ of $T$ on $\orth(E)$ is called the \term{diagonal} of $T$. If $T$ is positive, then Schep \cite[Theorem 1.1]{Schep80} proved that the projection $\mathcal P_E(T)$ is given by the formula

\begin{equation}\label{diagonal_formula}
\mathcal P_E(T)=\inf\{\sum_{i=1}^n P_iTP_i:\; 0\leq P_i\leq I, P_i^2=P_i, \sum_{i=1}^nP_i=I\}.
\end{equation}

Given a band $B$ of $E$ and a regular operator $T$ on $E$, we write shortly $\mathcal P_B(T)$ instead of
$\mathcal P_B(P_B T P_B)$. The operator $\mathcal P_B$ is the projection onto $\orth(B)$ of the order complete vector lattice $B$. We call the operator $\mathcal P_B(T)$ the diagonal of $T$ with respect to $B$.
The regular operator $\mathcal P_A(T)$ is the \term{atomic diagonal} of the operator $T$. It is not hard to see that for a positive operator $T$, we have
$$\mathcal P_A(T)=\sup\{\sum_{a\in \mathcal F}P_aTP_a:\; \mathcal F\subseteq \mathcal A \textrm{ is finite}\},$$
and that for a regular operator $T$, we have
$$\mathcal P_A(T)=\sum_{a\in \mathcal A}P_aTP_a$$
where the latter series converges in order. For more details, see \cite{Kandic13}.

The main concern of this section is to study those regular operators $T$ whose diagonal and atomic diagonal coincide, i.e., $\mathcal P_E(T)=\mathcal P_A(T)$. Schep proved \cite[Corollary 1.7]{Schep80} that for each atomless Banach lattice $E$ and for each positive compact operator $T\colon E\to E$, the zero operator is the only positive orthomorphism $S$ which satisfies $0\leq S\leq T$. In particular, if $E$ is order complete, this yields that for every positive compact operator, we have
$\mathcal P_E(T)=\mathcal P_A(T)$. Therefore, positive compact operators on order complete Banach lattices are examples of operators for which the diagonal and the atomic diagonal coincide. As already mentioned, we will extend this list by the class of abstract integral operators on order complete vector lattices (see \Cref{diagonal finite rank}) and positive Riesz operators $T$ on order complete Banach lattices (see \Cref{diagonal Riesz}).

If $E=B_1\oplus B_2$ is a band decomposition of $E$, then $\orth(E) \cong \orth(B_1) \oplus \orth(B_2)$. In particular, $\orth(E) \cong \orth(A)\oplus \orth(C)$.
The following lemma shows that every band decomposition of a vector lattice yields a decomposition of the diagonal operator.

\begin{lemma}\label{diagonal split}
If $E=B_1\oplus B_2$ is a band decomposition of an order complete vector lattice $E$, then $\mathcal P_E=\mathcal P_{B_1}\oplus \mathcal P_{B_2}$.
\end{lemma}

\begin{proof}
Let $T$ be any positive operator on $E$.
Suppose that $P_1,\ldots,P_n$ are  band projections satisfying $ 0\leq P_i\leq I, P_i^2=P_i, \sum_{i=1}^nP_i=I$. Since band projections commute, we have
\begin{align*}
\sum_{i=1}^n P_iTP_i & = (P_{B_1}+P_{B_2})\sum_{i=1}^n P_iTP_i(P_{B_1}+P_{B_2})\\
&\geq \sum_{i=1}^n P_{B_1}P_iTP_iP_{B_1}+\sum_{i=1}^n P_{B_2}P_iTP_iP_{B_2}\\
&=\sum_{i=1}^n (P_{B_1}P_iP_{B_1})P_{B_1}TP_{B_1}(P_{B_1}P_iP_{B_1})+\sum_{i=1}^n (P_{B_2}P_iP_{B_2})P_{B_2}TP_{B_2}(P_{B_2}P_iP_{B_2}) \\
& \geq  \mathcal P_{B_1}(T) + \mathcal P_{B_2}(T).
\end{align*}
By definition of $\mathcal P_E(T)$ we have $\mathcal P_E(T) \geq \mathcal P_{B_1}(T)+\mathcal P_{B_2}(T)$.
The opposite inequality can be shown similarly.
\end{proof}
%

The set of all regular operators for which diagonal operators coincide is always a band.

\begin{proposition}\label{P equals D on a band}
For each band $B$ of an order complete vector lattice $E$, the set
$$\{T\in \mathcal L_r(E):\; \mathcal P_E(T)=\mathcal P_B(T)\}$$
is a band in $\mathcal L_r(E)$.
\end{proposition}

\begin{proof}
Since $\orth(E) \cong \orth (B) \oplus \orth(B^d)$, the operator $\mathcal P_B\colon \mathcal L_b(E)\to \orth(B)$ is an order projection onto the band $\orth(B)$ in $\mathcal L_r(E)$. In particular, $\mathcal P_B$ is order continuous. Since $\mathcal P_B$ and $\mathcal P_E$ satisfy $0\leq \mathcal P_B \leq \mathcal P_E\leq \mathcal I$ where $\mathcal I$ denotes the identity operator on $\mathcal L_r(E)$, the difference $\mathcal P_E-\mathcal P_B$ is an orthomorphism on $\mathcal L_r(E)$. Hence, by \cite[Theorem 2.48]{Aliprantis:06}, the kernel of $\mathcal P_E-\mathcal P_B$ is a band in $\mathcal L_r(E)$ and the proof is concluded.
\end{proof}

The following corollary follows from \Cref{P equals D on a band} and the discussion preceding \Cref{diagonal split}.

\begin{corollary}\label{pre-Integral operators}
If $E$ is an order complete Banach lattice, then every operator $T\in (E\otimes E^\sim)^{dd}$ satisfies $\mathcal P_E(T)=\mathcal P_A(T)$.
\end{corollary}

\begin{proof}
Since the set of all regular operators $T$ for which $\mathcal P_E(T)=\mathcal P_A(T)$ is a band in $\mathcal L_r(E)$, it suffices to prove that $\mathcal P_E(T)=\mathcal P_A(T)$ for each positive rank-one operator $T$.
Since the rank of the operator $P_CTP_C$ is at most one, the inequality $0\leq \mathcal P_C(P_CTP_C)\leq P_CTP_C$ and \cite[Corollary 1.7]{Schep80} imply that $\mathcal P_E(T)-\mathcal P_A(P_ATP_A)=\mathcal P_C(P_CTP_C)=0$.
\end{proof}

Since $E_n^\sim\subseteq E^\sim$, \Cref{pre-Integral operators}, in particular, yields that every abstract integral operator $T$ on an order complete Banach lattice satisfies $\mathcal P_E(T)=\mathcal P_A(T)$. By the following theorem, the assumption of a complete norm structure is redundant.

\begin{theorem}\label{diagonal finite rank}
If $E$ is an order complete vector lattice, then every abstract integral operator $T$ satisfies
$\mathcal P_E(T)=\mathcal P_A(T)$.
\end{theorem}

\begin{proof}
Since the band of all abstract integral operators on $E$ is generated by the set of all positive order continuous rank one operators in $E\otimes E_n^\sim$, by \Cref{P equals D on a band} it suffices to prove that $\mathcal P_E(T)=\mathcal P_A(T)$ holds for every operator $T=x\otimes \varphi$ on $E$ with $0\leq \varphi\in E_n^\sim$ and $x\in E^+$.

Let us first consider the case when $\varphi$ is strictly positive (that is, $N_\varphi = \{0\}$) and $E$ is atomless. Let us define $\|x\|_\varphi:=\varphi(|x|)$. Then $\|\cdot\|_\varphi$ is a norm on $E$. If $\widetilde E$ is the norm completion of $E$, then
 $(\widetilde E,\|\cdot\|_\varphi)$ is an AL-space in which $E$ is embedded as a norm dense sublattice. Since $\varphi$ is order continuous and $E$ is order complete, an application of \cite[Theorem 4.1]{GTX} shows that $E$ is also an order dense ideal in $\widetilde E$.

We claim that $\widetilde E$ is atomless. If this were not the case, then there would exist an atom $\widetilde a$ in $\widetilde E$. Since $E$ is order dense in $\widetilde E$, there exists a non-zero positive vector $a$ in $E$ such that $0<a\leq \widetilde a$. If $0\leq x,y\leq a$ for some vectors $x,y\in E$ with $x\perp y=0$, then $x\perp y=0$ also in $\widetilde E$. Since $0\leq x,y\leq \widetilde a$, we have $x=0$ or $y=0$ proving that $a$ is an atom in $E$. This contradiction shows that $\widetilde E$ is atomless.

If $\widetilde B$ is a band in $\widetilde E$, then it is easy to see that $B:=\widetilde B\cap E$ is a band in $E$. We claim that $B$ is order dense in $\widetilde B$. If $\widetilde x\in \widetilde B$ is a positive vector, then there exists a net $(x_\alpha)$ of positive vectors in $E$ such that $0\leq x_\alpha\uparrow \widetilde x$ in $\widetilde E$. The inequality $0\leq x_\alpha \leq \widetilde x$ gives
$x_\alpha\in E \cap \widetilde B=B$, which proves the claim. From the order continuity of the norm on $\widetilde E$, it follows that $B$ is norm dense in $\widetilde B$.

To prove that $\mathcal P_E(x\otimes \varphi)=0$, by \eqref{diagonal_formula} it suffices to prove that the zero operator is the only positive operator $S$ on $E$ which satisfies $$0\leq S\leq \sum_{i=1}^n P_i(x\otimes \varphi)P_i$$ for all finite families of band projections $P_i$ which satisfy $\sum_{i=1}^nP_i=I$.
Let $\widetilde P_1,\ldots,\widetilde P_n$ be band projections on $\widetilde E$ which satisfy $\sum_{i=1}^n \widetilde P_i=\widetilde I$ where $\widetilde I$ denotes the identity operator on $\widetilde E$. Then $P_i:=\widetilde P_i|_{E}$ is the band projection onto the band, which is the intersection of the range of $\widetilde P_i$ with $E$. By assumption, we have $0\leq S\leq \sum_{i=1}^n P_i (x\otimes \varphi) P_i$, so that $S$ is dominated by a bounded operator on $E$. Hence, $S$ is also bounded, and let us denote by $\widetilde S$ its norm extension to $\widetilde E$. Let $\widetilde \varphi$ be the norm extension of $\varphi$ to
$\widetilde E$.
Since the norm extension of a positive bounded operator is again positive, it follows that
$$0\leq \widetilde S \leq \sum_{i=1}^n \widetilde P_i (x\otimes \widetilde \varphi) \widetilde P_i.$$
By \eqref{diagonal_formula} it follows that $0\leq \widetilde S\leq \mathcal P_{\widetilde E}(x\otimes \widetilde \varphi)$.
Since $x\otimes \widetilde \varphi$ is a positive rank-one operator on the atomless order complete Banach lattice $\widetilde E$, it follows that $\mathcal P_{\widetilde E}(x\otimes \widetilde \varphi)=0$ by \cite[Corollary 1.7]{Schep80}. This gives $\widetilde S=0$ and, therefore, $S=0$.

For the general case, observe first that by order continuity of $\varphi$ and order completeness of $E$, we have the band decomposition $E=C_\varphi\oplus N_\varphi$. The functional $\varphi$ is strictly positive on $C_\varphi$ and is zero on $N_\varphi$. This implies that $P_{N_\varphi}(x\otimes \varphi)P_{N_\varphi}=0$, and so
$\mathcal P_{N_\varphi}(P_{N_\varphi}(x\otimes \varphi){P_{N_\varphi}})=0$. This implies
\begin{align*}
\mathcal P_E(x\otimes \varphi)&=\mathcal P_E(P_{C_\varphi}(x\otimes \varphi)P_{C_\varphi}+P_{N_\varphi}(x\otimes \varphi)P_{N_\varphi})\\
&=\mathcal P_E(P_{C_\varphi}(x\otimes \varphi)P_{C_\varphi})\\
&=\mathcal P_A(P_AP_{C_\varphi}(x\otimes \varphi)P_{C_\varphi}P_A)+\mathcal P_C(P_CP_{C_\varphi}(x\otimes \varphi)P_{C_\varphi}P_C).
\end{align*}
Since $C$ is atomless and $\varphi$ is strictly positive on $C_\varphi \cap C$, the proof above shows that $\mathcal P_C(P_C P_{C_\varphi}(x\otimes \varphi)P_{C_\varphi} P_C)=0$. Since $x\otimes \varphi=0$ on $N_\varphi$,
we have $P_{A\cap N_{\varphi}}(x\otimes \varphi)P_{A\cap N_{\varphi}} = 0$, and so
\begin{align*}
\mathcal P_E(x\otimes \varphi)&=\mathcal P_A(P_AP_{C_\varphi}(x\otimes \varphi)P_{C_\varphi}P_A)\\
&=\mathcal P_A(P_{A\cap C_{\varphi}}(x\otimes \varphi)P_{A\cap C_{\varphi}}+P_{A\cap N_{\varphi}}(x\otimes \varphi)P_{A\cap N_{\varphi}})\\
&=\mathcal P_A(P_A(x\otimes \varphi)P_A)  \\
&=\mathcal P_A(x\otimes \varphi) .
\qedhere
\end{align*}
\end{proof}

We complete this paper by proving that $\mathcal P_E(T)=\mathcal P_A(T)$ for every positive Riesz operator on every order complete Banach lattice. Our proof relies on a subclass of the class of orthomorphisms. An operator $T\colon E\to E$ is \term{central} if there exists some scalar $\lambda>0$ such that $|Tx|\leq \lambda |x|$ for all $x\in E$. The set of central operators is denoted by $\mathcal Z(E)$. In the case when $E=\mathbb R^n$, a matrix $T$ is central if and only if it is diagonal.

By definition is clear that $\mathcal Z(E)$ is a subset of $\orth(E)$. Since central operators are order bounded and preserve disjointness, by \cite[Theorem 2.40]{Aliprantis:06} the modulus of a central operator $T$ on an Archimedean vector lattice $E$ exists and it satisfies \eqref{modulus} for each $x\in E$. This yields that $|T|$ is also central, and therefore, every central operator on an Archimedean vector lattice is a difference of two positive central operators. Wickstead \cite{Wickstead} provided an example when the inclusion of $\mathcal Z(E)$ in $\orth(E)$ is proper. He also proved in \cite[Proposition 4.1]{Wickstead} that bounded orthomorphisms on normed lattices are automatically central operators. Since every positive operator on a Banach lattice is bounded, for a Banach lattice $E$ we have $\mathcal Z(E)=\orth(E)$.

\begin{theorem}\label{diagonal Riesz}
Let $T$ be a positive Riesz operator on an order complete Banach lattice $E$. Then
$\mathcal P_E(T)=\mathcal P_A(T)$.
\end{theorem}

\begin{proof}
Since $T$ is a Riesz operator, by \cite[Theorem 3.12]{Dowson78}, the operator $T$ is an asymptotically quasi-finite-rank operator. This means that
$$\inf \{ \|T^n-F\|^{\frac{1}{n}}:\; F\in E\otimes E^{*}\} \to 0$$ as $n$ goes to infinity. Pick any $\epsilon>0$ and find $n_\epsilon$ such that
$$\inf\{ \|T^n-F\|^{\frac{1}{n}}:\; F\in E\otimes E^{*} \}<\epsilon$$ for all $n\geq n_\epsilon$.
Hence, for each $n\geq n_\epsilon$ there exists a finite rank operator $F_n$ such that
$\|T^n-F_n\|<\epsilon^n.$ By \cite[Theorem 4.14]{AA02}, the finite-rank operator $F_n$ has a compact modulus.
An application of \cite[Proposition 3.2]{DK14} then gives us that $\mathcal P_E(|F_n|)=\mathcal P_A(|F_n|)$, so that $\mathcal P_C(|F_n|)=0$.
The triangle inequality
$$0\leq T^n\leq |T^n-F_n|+|F_n|$$
and the fact that $\mathcal P_C$ is a lattice homomorphism yield that
$$0\leq \mathcal P_C (T^n)\leq \mathcal P_C(|T^n-F_n|).$$
Pick any $n\geq n_\epsilon$.
Since the norm and the regular norm of a central operator coincide, we have $\|\mathcal P_E (T^n-F_n)\|_r=\|\mathcal P_E (T^n-F_n)\|$.
Since the diagonal operator $\mathcal P_E$ is  a contraction with respect to the operator norm by \cite[Theorem 1.4]{Voigt88},
we have
\begin{align*}
|\mathcal P_C(T^n)\| & = \|\mathcal P_C(T^n)\|_r \leq \|\mathcal P_C(T^n-F_n)\|_r \leq \|\mathcal P_E(T^n-F_n)\|_r\\
& =\|\mathcal P_E(T^n-F_n)\|\leq \|T^n-F_n\|  <\epsilon^n.
\end{align*}
Since the diagonal operator $\mathcal P_C$ satisfies $\mathcal P_C(AB)\geq \mathcal P_C(A)\mathcal P_C(B)$ for positive operators $A$ and $B$, we have
$\mathcal P_C(T^n)\geq (\mathcal P_C(T))^n$, and so
$$ \|(\mathcal P_C(T))^n\| \leq \|\mathcal P_C(T^n)\|\leq \epsilon^n.$$
By \cite[Theorem 1.4]{Schep80}, the center $\mathcal Z(E)$ is an AM-space with a strong unit. Therefore, $\mathcal Z(E)$ is algebraically and lattice isometric to the Banach lattice algebra $C(K)$ of continuous functions on some compact Hausdorff space $K$. In particular, we have
$\|({\mathcal P}_C(T))^n\|=\|{\mathcal P}_C(T)\|^n$, from where it follows $\|{\mathcal P}_C(T)\|<\epsilon$. Since this holds for each $\epsilon>0$, we finally conclude that $\mathcal P_C(T)=0$.
\end{proof}

%
%



\end{document}